\documentclass[11pt]{article}
\usepackage{mathrsfs} 
\usepackage[margin=1.25in]{geometry}
\usepackage{ifthen}
\usepackage{graphicx}
\usepackage{epsfig}
\usepackage{dsfont, amssymb}
\usepackage{amsfonts,dsfont,amssymb,amsthm,stmaryrd,bbm}
\usepackage[cmex10]{amsmath} 

\usepackage{hyperref}
\hypersetup{colorlinks=true,pdftitle="",pdftex}

\newtheorem{conj}{Conjecture}[section]
\newtheorem{thm}{Theorem}[section]

\newtheorem{remark}[conj]{Remark}
\newtheorem{lem}[conj]{Lemma}

\newtheorem{defn}[conj]{Definition}
\newtheorem{cor}[conj]{Corollary}
\newtheorem{ex}[conj]{Example}

\newcommand\independent{\protect\mathpalette{\protect\independent}{\perp}} 
\def\independent#1#2{\mathrel{\rlap{$#1#2$}\mkern2mu{#1#2}}}

\newcommand{\rank}{\mathrm{rank}}

\newcommand{\bc}{\begin{center}}
\newcommand{\ec}{\end{center}}
\newcommand{\bt}{\begin{tabular}}
\newcommand{\et}{\end{tabular}} 
\newcommand{\bea}{\begin{eqnarray}}
\newcommand{\eea}{\end{eqnarray}}
\newcommand{\bean}{\begin{eqnarray*}}
\newcommand{\eean}{\end{eqnarray*}}

\newcommand{\ba}{\begin{array}}
\newcommand{\ea}{\end{array}}

\def\be{\begin{eqnarray}}
\def\ee{\end{eqnarray}}
\def\ben{\begin{eqnarray*}}
\def\een{\end{eqnarray*}}

\newcommand{\RL}{{\mathbb R}}

\newcommand{\esssup}{\mathop{\rm ess\, sup}}

\def\elabel#1{\label{e:#1}}

\def\sq{$\Box$}

\def\qed{\ifmmode\sq\else{\unskip\nobreak\hfil
\penalty50\hskip1em\null\nobreak\hfil\sq
\parfillskip=0pt\finalhyphendemerits=0\endgraf}\fi\par\medbreak}

\def\rank{{\rm rank\,}}

\newsavebox{\junk}
\savebox{\junk}[1.6mm]{\hbox{$|\!|\!|$}}

\def\limsup{\mathop{\rm lim\ sup}}
\def\liminf{\mathop{\rm lim\ inf}}

\def\til={{\widetilde =}}

\def\half{{\mathchoice{\textstyle \frac{1}{2}}%
{\frac{1}{2}}%
{\hbox{\tiny $\frac{1}{2}$}}%
{\hbox{\tiny $\frac{1}{2}$}} }}

 \def\eq#1/{(\ref{#1})}

\def\eq#1/{(\ref{e:#1})}

\newcommand{\beqn}[1]{\notes{#1}%
\begin{eqnarray} \elabel{#1}}

\newcommand{\eeqn}{\end{eqnarray} }

\newcommand{\beq}[1]{\notes{#1}%
\begin{equation}\elabel{#1}}

\newcommand{\eeq}{\end{equation}} 

\def\bdes{\begin{description}}
\def\edes{\end{description}}

\def\notes#1{}

\def\phi{\varphi}

\def\bee{\begin{eqnarray*}}
\def\ene{\end{eqnarray*}}

\begin{document}


\title{Rogozin's convolution inequality for locally compact groups}
\author{Mokshay Madiman, James Melbourne, Peng Xu\thanks{All the authors are with the 
Department of Mathematical Sciences, University of Delaware.
E-mail: {\tt madiman@udel.edu, jamesm@udel.edu, xpeng@udel.edu}.
This work was supported in part by the U.S. National Science Foundation through grants DMS-1409504 (CAREER) and CCF-1346564.}}
\date{}
\maketitle

\begin{abstract}
General extensions of an inequality due to Rogozin, concerning the essential supremum of a convolution of probability density functions on the real line, 
are obtained.  While a weak version of the inequality is proved in the very general context of Polish $\sigma$-compact groups, 
particular attention is paid to the group \(\mathbb{R}^d\), where the result can combined with rearrangement inequalities for 
certain linear images for a strong generalization. As a consequence, we obtain a unification and sharpening of both the 
\(\infty\)-Renyi entropy power inequality for sums of independent random vectors, due to Bobkov and Chistyakov, 
and the bounds on marginals of projections of product measures due to Rudelson and Vershynin (matching and extending the sharp improvement
of  Livshyts, Paouris and Pivovarov). The proof is elementary and relies on a characterization of extreme points of a class of 
probability measures in the general setting of Polish measure spaces, as well as the development of a generalization of 
Ball's cube slicing bounds for products of \(d\)-dimensional Euclidean balls (where the ``co-dimension 1" case had been recently settled by Brzezinski).
\end{abstract}

\tableofcontents

\newpage




\section{Introduction}
\label{sec:intro}

Throughout we will reserve $E$ to denote a Polish measure space with its Borel $\sigma$-algebra, which admits a 
regular\footnote{By this we mean that Borel sets can be approximated in measure both by compact sets from within, and by surrounding open sets.}
$\sigma$-finite measure $\alpha$. Fix this measure, which we will call the 
{\it reference measure} on $E$. The law\footnote{By this we mean $\mathbb{P}(X \in A) = \mu(A)$} of a given $E$-valued random variable $X$ defines a Borel probability measure $\mu$ on $E$.  We denote by $\mathcal{P}(E)$ the space of all such $\mu$, and give the space the topology induced by weak-* convergence, and more generally for measurable $S \subseteq E$, $\mathcal{P}(S) = \{ \mu \in \mathcal{P}(E) : supp(\mu) \subseteq S \}$. We define $M: \mathcal{P}(E) \to [0,\infty]$ for by 
\ben
	M(\mu) = \sup \left\{\frac{\mu(A)}{\alpha(A)}: \alpha(A)>0 \right\} .
\een
When $X$ is a random variable distributed according to $\mu$, we will often use $M(X)$ in place of $M(\mu)$.
We will write $\mathcal{P}_m(E) = M^{-1}([0,m])$, and for $S \subseteq E$, $\mathcal{P}_m(S)$ will denote $\{ \mu \in \mathcal{P}_m(E) : supp(\mu) \subseteq S \}$.
If $M(X)$ is finite, the Radon-Nikodym theorem guarantees the existence of a density $f$ with respect to $\alpha$, in which case
\ben
M(X) =    \esssup_x f(x)=\|f\|_{L^\infty(E)}.
\een

Often we will take $E$ to be a locally compact group, denoted $G$ in this case, and the reference measure $\alpha$ to be 
the left Haar measure.  
Our first main result is cast in the case $G=\mathbb{R}^d$, with the 
reference measure being Lebesgue measure. 

\begin{thm}\label{thm:Generalized infinity EPI}
	For a random variable $X = (X_1, \dots X_n)$, composed of $X_i$ independent $\mathbb{R}^d$ valued vectors, $P$ a projection matrix from $\mathbb{R}^n$ to a $k$ dimensional subspace, $I_d$ the $d$-dimensional identity matrix, and $P \otimes I_d$ denoting their Kronecker product, there exists $\gamma_i \in [0,1]$  dependent only on $P$ such that $\sum_{i=1}^n \gamma_i = k$ and
	\[
	M(P \otimes I_{d}(X)) \leq c(d,k) \prod_{i=1}^n M^{\gamma_i}(X_i)
	\]
	where 
	\begin{equation*}
	c(d,k) = \begin{cases}
	\min\left\{\left(\frac{(1 + \frac d 2 )^{\frac{d}{2}}}{\Gamma(1 + \frac d 2)} \right)^k , \left(\frac{n} {n-k} \right)^{\frac{d(n-k)}{2}} \right\} &\text{$d \geq 2$} \\
	\min \left\{2^{\frac k 2}, \left(\frac{n}{n-k} \right)^{\frac{n-k}{2}}\right\} &\text{$d=1$}.
	\end{cases}
	\end{equation*}
\end{thm}

Let us discuss some special cases and consequences.  When $k=1$, the result should be compared to the following result due to Bobkov and Chistyakov.

\begin{thm}\label{thm:bc14}\cite{BC14}
	When $X_1, \cdots, X_n$ are independent random vectors in $\mathbb{R}^d$,
	\begin{equation}\label{eq:Bobkov and Chistyakov}
		M^{-\frac 2 d}(X_1 + \cdots + X_n) \geq \frac 1 e \sum_{i=1}^n M^{-\frac 2 d} (X_i).
	\end{equation}
\end{thm}
Indeed, let us briefly sketch the relationship.  Assume without loss of generality that the $X_i$ are bounded and write
\[
	X_1 + \cdots + X_n = k (\theta_1 \tilde{X}_1 + \cdots + \theta_n \tilde{X}_n ),
\]
with $\tilde{X}_i = M^{-\frac 1 d}(X_i)X_i$, and $\theta_i = M^{-\frac 1 d}(X_i)/ \sqrt{ \sum_j M^{- \frac 2 d}(X_j)}$, so that $\theta$ is a unit vector, $M(\tilde{X}_i)=1$ and consequently
\[
	k = \sqrt{ \sum M^{-\frac 2 d}(X_j)}.
\] 
Taking maximums, we compute.
\begin{align*}
	M(X_1 + \cdots + X_n) 	
		&=
			k^{- d} M(\theta_1 \tilde{X}_1 + \cdots + \theta_n \tilde{X}_n)
				\\
		&=
			k^{- d} M(P_\theta \otimes I_{d}(\tilde{X})),
\end{align*}
where $\tilde{X} = (\tilde{X}_1, \dots, \tilde{X}_n)$, and $P_\theta$ denotes the projection matrix to the vector $\theta$.  Applying Theorem \ref{thm:Generalized infinity EPI}, and unwinding a small bit of algebra we arrive back at equation \eqref{eq:Bobkov and Chistyakov}
\[
	M^{-\frac 2 d}(X_1+ \cdots + X_n) \geq c(d,1)^{-\frac 2 d} \sum_{i=1}^n M^{-\frac 2 d} (X_i),
\]
The constant can be written explicitly as $\Gamma^{\frac 2 d}(1 + \frac d 2)/(1 + \frac d 2)$, which using Stirling's approximation can be shown to decrease with $d$ to $\frac 1 e$.  Independent of the number of summands, these constants will be shown optimal.

In information theory, the number $h_{\infty}(X)=-\log M(X)$ (defined to be $-\infty$ in the case that $M(X)$ is infinite) 
is called the $\infty$-R\'enyi entropy (or {\it R\'enyi entropy of infinite order})
of $X$, and the number $N_\infty(X) = e^{2h_{\infty}(X)/d}= M(X)^{-2/d}$ is called the {\it $\infty$-R\'enyi entropy power}.  
The $\infty$-R\'enyi entropy is one instance of a one-parameter family of
R\'enyi entropies $h_p(X)$ for $p\in [0,\infty]$, which include the Shannon-Boltzmann
entropy $h_1(X)=-\int_E f(x)\log f(x) \alpha(dx)$.
Many of the results within this paper can be re-interpreted in this language.  

Inequalities for the one-parameter family of R\'enyi entropies 
play an important role in several areas of mathematics including
probability \cite{Sta59, Bar86, ABBN04:1, Yu09:1, JKM13, BJKM10}, 
information theory \cite{Sha48, Ber74, Cos85a},  
additive combinatorics \cite{MMT12, Tao10, MK17, Hoc14:1}, and
convex geometry \cite{LYZ04, BM11:cras, BM13:goetze, FMW16}; 
this literature restricted to the Euclidean setting is surveyed and extended in \cite{MMX17:0}.
Furthermore, inequalities for the $\infty$-R\'enyi entropy (particularly as applied to sums of random variables taking values in a group)
in particular are closely related to the study of small ball probabilities \cite{LS01, NV13, EGZ15}, which is important
both from the viewpoint of probability and additive combinatorics,
and is closely connected to recent developments in random matrix theory.

Theorem~\ref{thm:Generalized infinity EPI} will allow us to prove the $\infty$-entropy power inequality
\be\label{inq:Main}
N_\infty\left(\sum_{i=1}^nX_i\right)
\ge \frac{\Gamma^{\frac 2 d}\left(\frac{d}{2} + 1 \right)}{\frac d 2+1}\sum_{i=1}^nN_\infty(X_i),
\ee
which clearly sharpens Theorem~\ref{thm:bc14}.
Further connections of the results here with Information Theory are developed in \cite{XMM17:isit:1, XMM17:isit:2}.

In a different direction, important recent developments in
random matrix theory \cite{RV09, TV09:1, RV10:icm} motivated Rudelson and Vershynin's investigation into small ball 
properties of linear images of high-dimensional distributions.  To this end, the main result of \cite{RV15} was the following.

\begin{thm}\label{thm:RV}\cite{RV15}
	When $X = (X_1, \dots, X_n)$ is composed of $\mathbb{R}$ valued independent random variables with $M(X_i) \leq L$ and $P$ a projection matrix to a $k$ dimensional subspace, 
	\[
	M(PX) \leq (C L)^k
	\]
	for some $C>0$.
\end{thm}

This was optimally sharpened soon after by Livshyts, Paouris, and Pivovarov \cite{LPP15} and it is exactly this sharpening that we recover in 
Theorem \ref{thm:Generalized infinity EPI} when $d=1$.  Thus we can see Theorem \ref{thm:Generalized infinity EPI} as a 
generalization of both Theorems~\ref{thm:bc14} and \ref{thm:RV}, unifying the two perspectives, and 
including the optimal summand-independent sharpening of \eqref{eq:Bobkov and Chistyakov}.

We consider our methods to be of independent interest. Their origin may be traced to 
the 1987 paper of Rogozin \cite{Rog87:1}, where the following elegant result was obtained for the case $(E,\alpha) = (\mathbb{R}, dx)$. If 
$X_1, \dots, X_n$ are independent real-valued random variables,
and $U_1, \dots, U_n$ are independent random variables with uniform
distributions on intervals with widths determined by setting $M(U_i)=M(X_i)$,
then
\be\label{eq:rog}
M(X_1 + \cdots + X_n ) \leq M(U_1 + \cdots + U_n) .
\ee
The lack of a multidimensional analog of Rogozin's inequality \eqref{eq:rog} 
had been noted by a number of authors. Rudelson and Vershynin \cite{RV15} say, for instance:
``A natural strategy would be to extend to higher dimensions the simple one-dimensional
argument leading to Theorem 1.2, which was a combination of Ball's and
Rogozin's theorems. [{\color{red}...}] However, we are unaware of any higher-dimensional versions of
Rogozin's theorem.''\\


We present the following extension of Rogozin's theorem.

\begin{thm}\label{thm:GRGZ}
Let $X =(X_1, \dots, X_n)$ have components $X_i$ independent $\mathbb{R}^d$-valued random vectors,
and let $Z= (Z_1, \dots, Z_n)$ be comprised of $Z_i$ independent random vectors distributed uniformly on origin-symmetric Euclidean balls in $\mathbb{R}^d$ 
whose radii are determined by the equations $M(Z_i) = M(X_i)$ (with the understanding that $Z_i$
is a Dirac delta distribution at $0$ if $M(X_i)=\infty$). 
Then
\be\label{inq:GRGZ}
	M(T \otimes I_d (X) ) \leq M(T \otimes I_d(Z)) .
\ee
where $T: \mathbb{R}^n \to \mathbb{R}^k$ is a linear map, and $I_d$ is the identity map on $\mathbb{R}^d$.
\end{thm}

As mentioned above, the map $T \otimes I_d$ is sometimes known as the Kronecker product of $T$ and $I_d$.  It can be understood, 
either by the tensor product identification of $(\mathbb{R}^{d})^n$ with $\mathbb{R}^n \otimes \RL^d$ 
via $(x_1,\cdots x_n) = \sum_{i=1}^n e_i \otimes x_i$, where $e_i$ denotes the standard basis for $\mathbb{R}^n$, 
or by the following formula.  Write $T$ with respect to the standard basis $\{e_i\}$ as $T_{ij} = e_i \cdot T e_j$, 
then $(T\otimes I_d(X))$ in $(\RL^d)^n$ has the coordinates 
\[
	(T\otimes I_d(x_1, \dots, x_n))_i = \sum_{j=1}^n T_{ij} x_j.
\]
We will make frequent enough use of these tensor maps that we employ the abbreviated notation $T^{(d)} = T \otimes I_d$.

In particular, when $T$ is a projection map to a one dimensional subspace, Theorem \ref{thm:GRGZ} 
asserts that the inequality \eqref{eq:rog} holds for random vectors $X_i$ in $\RL^d$, with each $U_i$
being uniformly distributed on a Euclidean ball with $M(U_i)=M(X_i)$.
After the first version of this paper was presented, we learnt from S. Bobkov that this particular
case had recently been obtained by Ju{\v{s}}kevi{\v{c}}ius and Lee \cite{JL15}, and we also learnt from
P. Pivovarov and G. Livshyts that it follows from stochastic ordering results obtained by Kanter \cite{Kan76:1, Kan77}.
Apart from extending Rogozin's theorem, this inequality is related to a family of inequalities for R\'enyi entropy
of arbitrary order $p$, conjectured by Wang and the first-named author \cite{WM14}, with a certain class 
of generalized Gaussian distributions indexed by a parameter $\beta_p$
being the conjectured extremizers. For $p=\infty$, the corresponding generalized Gaussian distribution
turns out to be the uniform distribution on a Euclidean ball. 
The inequality \eqref{eq:rog} for random vectors in $\RL^d$ confirms  \cite[Conjecture IV.3]{WM14} 
when $p=\infty$.

%

Theorem \ref{thm:GRGZ} is proven in two parts. The first reveals a very general phenomenon that 
does not require much of the structure present in our situation. 
\begin{thm}\label{thm:MainLC}
	Let $\{(E_i,\alpha_i)\}_{i=1}^n$, be Polish measure spaces with regular Borel measures $\alpha_i$.
For any coordinate-convex, lower semi-continuous (with respect to the product of weak-* topologies) function
	$
	\Phi: \prod_{i=1}^n \mathcal{P}_{m_i}(E_i) \to \mathbb{R},
	$
we have
	\begin{align}\label{inq:MainLC}
	\sup_{(\mu_1, \ldots, \mu_n)\in \prod_{i=1}^n \mathcal{P}_{m_i}(E_i) } \Phi(\mu_1, \ldots, \mu_n)  
	&= 
	\sup_{(\mu_1, \ldots, \mu_n) \in\prod_{i=1}^n \mathcal{P}_{m_i}^*(E_i)} \Phi(\mu_1, \ldots, \mu_n) ,
	\end{align}
where $\mathcal{P}_{m}^*(E)$ denotes the space of Dirac measures on $E$ if $m=\infty$,
and the space of measures $\mu$ with densities of the form
	\[
	\frac{d\mu}{d\alpha}(y) = m \mathbbm{1}_U(y) + c \mathbbm{1}_{\{x\}}(y) ,
	\]
for some $U$ contained in a compact subset of $E$, some $c \in [0,m)$ and some $x \notin U$, 
when $m$ is finite. 
\end{thm}

For example, we have the following weak Rogozin theorem for locally compact groups.  
Given a group $(G,\cdot)$, we will write $xy$ for $x \cdot y$, $e$ for the identity element of a group, and  $x^{-1}$ for the inverse of $x$ under the group operation.

\begin{thm}\label{cor:MainLC}
Suppose $(G,\cdot)$ is a Polish $\sigma$-locally compact group with a non-atomic left Haar measure $\alpha$. Let $X_1,\cdots, X_n$ be independent $G$-valued random variables with density functions $f_1,\cdots, f_n$.
Then the following inequality holds
\be\label{inq:MainLCsum}
M(X_1\cdots X_n)\le \sup M(U_1\cdots U_n)
\ee
where the supremum is taken over all independent $G$-valued random variables $U_1,\cdots, U_n$ with $U_i$ uniform distributed on a subset of $G$ according to $M(U_i)=M(X_i)$ (with the understanding that $U_i$ is a Dirac delta distribution if $M(X_i)=\infty$).
\end{thm}

The second step in the proof will be a general rearrangement theorem for maxima of Euclidean densities.  It allows one to dispense with the 
supremum in the inequality \eqref{inq:MainLCsum}, reducing the random variable problem to a particular convex geometric one.  
It is here that our techniques are limited to the Euclidean setting and certain linear maps.  When dealing with $\RL^d$,  
there is a particularly nice notion of rearrangements for functions on $\RL^d$
(namely, spherically symmetric, decreasing rearrangements) with a well developed set of rearrangement inequalities
that can be deployed, which will allow us to arrive at the generalization described in Theorem \ref{thm:GRGZ}.

Our approach is actually similar in philosophy to Rogozin's, in that we exploit an argument based
on an understanding of extreme points of an appropriate class of measures. However Rogozin's argument contained two steps: a discretization
step and an extreme point characterization in the discrete setting. By eschewing the
discretization and directly implementing an appropriate extreme point characterization
in the general setting of Polish measure spaces, we simultaneously simplify 
Rogozin's argument and allow for significant extension of his results.

In contrast, when discussing their methods for studying anticoncentration for linear images
of high-dimensional distributions, which are based on an unpublished paper
of Ball and Nazarov \cite{BN96:unpub} and do not yield sharp constants, 
Rudelson and Vershynin \cite{RV15} say ``this approach avoids the delicate combinatorial
arguments that appear in the proof of Rogozin's theorem''; Bobkov and Chistyakov \cite{BC14}
similarly characterize Rogozin's theorem as ``delicate''. What follows will demonstrate Rogozin's 
inequality and the generalizations that we develop as a consequence of general and simple mechanisms. 

With this in hand, we will prove Theorem \ref{thm:Generalized infinity EPI} by proving the inequality in the special case that the $X_i$ are 
uniformly distributed on origin symmetric balls.

In parallel to the Euclidean space discussion, we will also develop an analogous (partial) theory for integer-valued random variables with their 
usual Haar (counting) measure.  To complete the development of a theory in the integer setting requires additional
ingredients, and such a completion will be reported in a forthcoming paper. 

We will organize this paper as follows. In Section~\ref{sec:extreme}, we will first describe applications of Theorem~\ref{thm:MainLC}
before deploying the Krein-Milman machinery and a general approximation argument to obtain its proof.

In Section~\ref{sec:rearrangement},  we first deduce Theorem \ref{thm:GRGZ} from Theorem \ref{cor:MainLC} by proving 
a generalization of the rearrangement inequality of \cite{WM14}, corresponding to the $\infty$-R\'enyi entropy case.  This in turn relies on the 
Rogers-Brascamp-Lieb-Luttinger rearrangement inequality \cite{Rog57, BLL74}.

In Section~\ref{sec:Quantitative bounds}, we will derive the quantitative bounds of Theorem~\ref{thm:Generalized infinity EPI}.  
Separate arguments will be put forth for the constants appearing in the minima, one bound derived from straightforward analysis 
of the density function the releveant random variable, the other from investigation of its characteristic function.  Both results 
will rely on application of a geometric Brascamp-Lieb inequality. 

\section{Extreme points of a class of measures}

\label{sec:extreme}


\subsection{Examples}
Let us first discuss some immediate consequences of the Theorem \ref{thm:MainLC}, starting with a more explicit description of $\mathcal{P}_m^*(E)$ in two general cases.  We will following this with some examples of convex lower semi-continuous functions on $\prod_{i=1}^n \mathcal{P}_{m_i}(E_i)$.  
\begin{enumerate}
\item
When $\alpha$ is non-atomic,  $\mathcal{P}_m^*(E)$ are compactly supported measures with density of the form $m \mathbbm{1}_U$, where $\alpha(U)  = \frac 1 m$.

\item \label{case2}
When $\alpha$ is discrete, an element of $\mathcal{P}_m^*(E)$  is necessarily supported on a finite set, say $X = \{x_0, x_1,  \dots x_n \}$ in $E$ and has a density realized as 
\[
	m \mathbbm{1}_{X \setminus x_0} + c \mathbbm{1}_{x_0}.
\]
such that $m\alpha(X) > 1$, and $c = \frac{1- m \alpha( X \setminus x_0)}{\alpha(x_0)} $. In particular when $\alpha$ is a counting measure, $X$ can be written as $\{x_0, x_1, \dots, x_{[\frac 1 m]} \}$, and $c = 1 - m [\frac 1 m]$.
\end{enumerate}

Interesting examples can be built by composing couplings with functionals.
\begin{ex} {\normalfont
Given a metric $c$, on a compact space $K$, the map $C(\mu, \nu) = \inf_{\pi \in \Pi(\mu,\nu)} \int_{K \times K} c(x,y) \pi(dx, dy)$ is convex and lower semi-continuous.  Here $\Pi(\mu,\nu)$ is the space of all couplings of $\mu$ and $\nu$, in the sense that $\pi \in \mathcal{P}(K \times K)$ and its marginals satisfy $\pi_1 = \mu$, $\pi_2 = \nu$, see \cite{Vil03:book}. } 
\end{ex}

The majority of this paper will actually be concerned with the trivial coupling.  The following example will be of later use.

\begin{lem}\label{lem:push}
Let $(E_1,\alpha_1)$, $\cdots$, $(E_n, \alpha_n)$ and $(E,\alpha)$ be Polish measure spaces.  For a continuous map $\varphi: E_1 \times \cdots \times E_n \to E$, and a convex lower semi-continuous functional $T: \mathcal{P}(\prod_{i=1}^n E_i) \to \mathbb{R}$ the map, 
	\[
		(\mu_1, \dots, \mu_n) \mapsto T(\varphi_*(\mu_1 \otimes \cdots \otimes \mu_n))
	\]
is lower semi-continuous and coordinate convex.
\end{lem}

\begin{proof}
For convexity it suffices to show $(\mu_1, \dots, \mu_n) \mapsto \varphi_*(\mu_1 \otimes \cdots \otimes \mu_n)$ is affine in its coordinates, as its composition with the convex map $T$, is then necessarily coordinate convex.
For clarity and notational brevity we will consider only the $n=2$ case, as the general situation is no more complicated.  To prove affineness, take Borel $A \subseteq E$, and observe that from the definition of the pushforward and the product measure,
\begin{align*}
			 (\lambda \mu_1 + (1-\lambda)\nu_1) \otimes \mu_2 (\varphi^{-1}(A))
		&=
			\lambda \mu_1 \otimes \mu_2(\varphi^{-1}(A)) + (1-\lambda) \nu_1 \otimes \mu_2(\varphi^{-1} (A))
\end{align*}
Thus $\varphi_*(\lambda \mu_1 +(1-\lambda) \nu_1 , \mu_2 ) = \lambda \varphi_*(\mu_1 , \mu_2) + (1-\lambda)\varphi_*(\nu_1, \mu_2)$, and since the argument for the second coordinate is identical, convexity is proved.

For continuity, we will show that we have a composition of continuous maps.  For $f$ continuous and bounded , $f\circ \varphi$ is as well.  Thus for $\lim_n \mu_n = \mu$, $ \lim_n \mu_n( f\circ \varphi) = \mu(f\circ \varphi)$.  Thus the pushforward mapping induced by a continuous map is continuous with respect to the weak-* topology.

The map $(\mu, \nu) \mapsto \mu \otimes \nu$ is continuous as well, as product measure convergence in the weak-* topology is equivalent to weak-* convergence of the marginals, see for instance \cite{Bil99:book}.  Or if one fixes metrics $d_1$ and $d_2$ on $E_1$ and $E_2$ lets $W_{1,d}$ denote the $1$-Kantorovich-Rubenstein distance with respect to a metric $d$, it is straight forward to prove that
\[
	W_{1,d}(\mu \otimes \nu, \mu' \otimes \nu') \leq W_{1,d_1} (\mu, \nu) + W_{1,d_2}(\mu',\nu')
\]
where $d((x_1,x_2),(y_1,y_2)) = d_1(x_1,y_1) + d_2(x_2,y_2)$.  Thus we have lower semi-continuity as this property is preserved by composition with continuous functions.
\end{proof}

The following examples are then corollaries.

\begin{ex}[Littlewood Offord] {\normalfont
When $X_i$ are independent random variables taking values in a compact subset $K_i$ of a ring $R$ with a reference measure $\alpha_i$. Then for fixed $v_i$ in a topological $R$-module $M$ with and an open subset $\mathcal O$,
\[
	 \mathbb{P}(v_1 X_1 + \cdots + v_n X_n \in \mathcal{O}) \leq \sup_U  \mathbb{P}(v_1 U_1 + \cdots + v_n U_n \in \mathcal{O})
	\]
	where the suprema is taken over all $U_i$ independent with law belonging to $\mathcal{E}(\mathcal{P}_{\alpha_i}(K_i))$. }
\end{ex}

\begin{proof}
Taking  $T: \mathcal{P}(M) \to \mathbb{R}$, defined by $\mu \mapsto \mu(\mathcal{O})$ is lower semi-continuous and linear.  The map $a: {R}^n \to M$ given by $a(x_1,\dots, x_n) = x_1 v_1 + \cdots + x_n v_n$ is continuous. Thus by the map $:\prod_i^n \mathcal{P}_{\alpha_i}(K_i) \to \mathbb{R}$ defined by $(\mu_1, \dots, \mu_n) \mapsto T v_*( \mu_1 \otimes \cdots \otimes \mu_n ).$ is coordinate convex (it is actually multi-linear) and lower semi continuous.  Thus by Lemma \ref{lem:push},
\begin{align*}
	T v_*(\mu_1, \dots, \mu_n) &\leq \sup_{\nu_i \in \mathcal{E}(\mathcal{P}(K_i))}T a_*(\nu_1, \dots \nu_n))
\end{align*}
But writing $X_i \sim \mu_i$, his is exactly
\begin{align*}
	\mathbb{P}(a_1 X_1 + \cdots + a_n X_n \in \mathcal{O}) 
		\leq 
			\sup_U  \mathbb{P}(a_1 U_1 + \cdots + a_n U_n \in \mathcal{O})
\end{align*}
\end{proof}

\begin{ex} {\normalfont
Suppose that $X_1, \dots, X_n$ are independent $G$-valued random variables on a compact group $G$ with Haar measure $\alpha$.  Then, 
\[
	M (X_1 \cdots X_n) \leq \sup M (U_1 \cdots U_n),
\]
when $\alpha$ is atom-free, the supremum is taken over all $U_i$ uniformly distributed on sets such that $M(U_i) = M(X_i)$, otherwise $\alpha$ is a counting measure and $U_i$ is of the form described in case \ref{case2} above. }
\end{ex}
We will pay particular attention to the Euclidean case on $\mathbb{R}^d$ and the integer case on $\mathbb{Z}$.
\begin{ex} \label{ex:eulideanrogozin} {\normalfont
When $X = (X_1, \dots, X_n)$ has coordinates $X_i$ bounded, independent $\mathbb{R}^d$ valued random  variables, and $P$ is a projection from $\mathbb{R}^n$ to a $k$-dimensional subspace,
\begin{align*}
	M(T^{(d)} (X)) \leq \sup_U M(T^{(d)}  (U)),
\end{align*}
where the suprema is taken over all $U = (U_1, \dots, U_n)$ with independent coordinates $U_i$, uniform on sets chosen satisfying $M(X_i) = M(U_i)$. }
\end{ex}

\begin{ex} \label{ex:integerrogozin} {\normalfont
When $X = (X_1, \dots, X_n)$ has coordinates $X_i$ bounded, independent, and taking values in a discrete group $G$, with counting reference measure.  Then,
\begin{align*}
	M(X_1  \cdots  X_n) 
		&\leq 
			\sup_U M(U_1  \cdots  U_n)
				\\
		&\leq 
			\sup_V M(V_1  \cdots  V_n)
\end{align*}
where the first suprema taken over all $U = (U_1, \dots, U_n)$ with independent coordinates $U_i$, nearly uniform on sets in the sense of case \ref{case2} such that $M(X_i) = M(U_i)$, and the second is taken over all $V = (V_1, \dots, V_n)$ with independent coordinates $V_i$ uniform on sets of size determined by the equation $M(V_i) = [M(X_i)^{-1}]^{-1}$.
}
\end{ex}
Note that we use the notation $[x]$ for the floor of $x$ equal to $\sup \{n \in \mathbb{Z} : n \leq x\}$.  The second inequality follow from the fact that $[M(X_i)^{-1}]^{-1} \geq M(X_i)$, so we can take a supremum over a larger set.


\subsection{Proof of Extremizers}
In order to recall the classical Krein-Milman theorem, let us fix some notation. For $S$ a subset of a locally convex space, we will denote by $\mathcal{E}(S)$ the set of {\it extreme points}\footnote{That is the set of points $s \in S$ such that $x,y \in S$ with $\half(x+y) = s$ implies $x=y=s$} of $S$
. 
We will denote by $co(S)$ the convex hull of $S$ defined to be all finite convex combinations of elements in $K$, and $\overline{co}(S)$ for its topological closure.

\begin{thm}\label{thm:KM}(\cite{KM40}) For a locally convex topological vector space $V$, 
if $K\subset V$ is compact and convex, then $K=\overline{co}(\mathcal{E}(K))$.
\end{thm}
The significance of the result for our purposes is that it reduces the search for suprema in the following case, if $\Phi: K \to \mathbb{R}$ is convex and lower semi-continuous in the sense that $x_\beta \to x$ implies, $\Phi(x) \leq \liminf \Phi(x_\beta)$, then we can limit our search for extremizers to the smaller set $\mathcal{E}(K)$ as
\[
	\sup_{x \in K} \Phi(x) = \sup_{x \in \mathcal{E}(K)} \Phi(x).
\]
Indeed, $x \in K$, can be written as the limit of $x_\beta$ in $co(\mathcal{E}(K))$, and any such $\beta$ can be written as a finite convex combination of extreme points, $x_\beta = \sum \lambda_i e_i$, but by convexity any such elements value can be bounded by $\Phi(x_\beta) \leq \sum \lambda_i \Phi(e_i) \leq \sup_{\mathcal{E}(K)} \Phi$.  Thus
\begin{align*}
	\Phi(x) 
		&\leq 
			\liminf \Phi(x_\beta) 
			\\
		&\leq  
			\sup_{x \in \mathcal{E}(K)} \Phi(x).
\end{align*}

Recall from Prokhorov's tightness theorem that for  a compact subset $K$ of a Polish space,  $\mathcal{P}(K)$ is a compact subset of the (Polish vector space) of all finite Borel signed measures endowed with the weak-* topology.  Since $\mathcal{P}(K)$ is clearly convex,  by Krein-Milman $\overline{co}(\mathcal{E}(\mathcal{P}(K))) = \mathcal{P}(K)$, and in fact it is straightforward to show that $\mathcal{E}(\mathcal{P}(K))$ is the set of Dirac measures on $K$.  The main purpose of this subsection is to prove the following lemma, which shows that the Krein-Milman framework applies to $\mathcal{P}_m(K)$ and thus simplifies the search for maximizers of certain functionals.

\begin{lem}\label{lem:MainLC}
	Let $\{(E_i,\alpha_i)\}_{i=1}^n$, be Polish metric measure spaces with regular Borel measures $\alpha_i$, then for compact $K_i  \subseteq E_i$
 any coordinate-convex, lower semi-continuous with respect to the product of weak-* topologies,
	$
	\Phi: \prod_{i=1}^n \mathcal{P}_{m_i}(K_i) \to \mathbb{R}
	$
	satisfies
	\begin{align}\label{inq:MainLC}
	\sup \Phi(\prod_{i=1}^n \mathcal{P}_{m_i}(K_i)) 
	&= 
	\sup \Phi\left(\mathcal{E}\left(\prod_{i=1}^n \mathcal{P}_{m_i}(K_i)\right)\right) 
	\\
	&= 
	\sup \Phi \left(\prod_{i=1}^n \mathcal{E}(\mathcal{P}_{m_i}(K_i))\right).
	\end{align}
	Furthermore, the extreme points  $\mu\in \mathcal{E}(\mathcal{P}_m(K))$ are of the form
	\[
	\mu(A) = \int_A m \mathbbm{1}_U(y) + c \mathbbm{1}_{\{x\}}(y) d\alpha(y),
	\]
	for a measurable subset $U$ of  $K$, $c \in [0,m)$ and $x \in K \setminus U$, and for any measurable set $A$.
\end{lem}

Portions of the proof of Lemma \ref{lem:MainLC} will be useful to us and we isolate these results as lemmas.  In particular, we demonstrate the lower semi-continuity of $M$ in Lemma \ref{lem:LSC}, that $\mathcal{P}_m(K)$ is compact and convex in Lemma \ref{lem:topo} and describe its extreme points as Lemma \ref{lem:extremes}.  
 
\begin{lem}\label{lem:LSC}
	The map $M:\mathcal{P} (E) \to [0,\infty]$
	is lower semi-continuous (with respect to the weak-* topology).
\end{lem}
\begin{proof}
	Assume $\mu_\beta \to \mu$, and let $A$ be such that $\alpha(A)>0$, for open $\mathcal{O}$ containing $A$, by the definition of weak-* convergence\footnote{That is the locally convex
Hausdorff topology generated by the separating family of seminorms $\|\mu \|_f:=|\int f d\mu |$,
where $f$ ranges over all continuous functions with compact support} and $M$
	\begin{align*}
		\mu(\mathcal{O}) \leq \liminf_\beta \mu_\beta(\mathcal{O}) \leq \alpha(\mathcal{O}) \liminf_\beta M (\mu_\beta).
	\end{align*}
By the inclusion hypothesis, and the regularity of $\alpha$,
\[
	\mu(A) \leq \alpha(A) \liminf_\beta M(\mu_\beta).
\]
As $A$ was arbitrary, we have our result $M(\mu) \leq \liminf_\beta M (\mu_\beta)$.
\end{proof}

\begin{lem} \label{lem:topo}
For any $S \subseteq E$,  $\mathcal{P}_m(S)$ is convex.  When $S=K$ compact, $\mathcal{P}_m(K)$ is compact in the weak-* topology.
\end{lem}
\begin{proof}
For $\mu, \nu \in \mathcal{P}_m(S)$ and $t \in (0,1)$, the computation
\[
	((1-t)\mu + t \nu)(A) \leq (1-t) m \alpha(A) + t m \alpha(A) = m \alpha(A),
\]
shows that $\mathcal{P}_m(S)$ is convex.  By Prokhorov's compactness theorem, the space of all Borel measures supported on $K$ is compact and since $\mathcal{P}_m(K) = M^{-1}(-\infty,m]$, it is a closed subset (by the lower semi-continuity of $M$ proved in Lemma \ref{lem:LSC}) of a compact metrizable set.
\end{proof}

\begin{lem} \label{lem:extremes}
$\mathcal{P}_m(K)$ (for all $m\in (0,\infty]$) is convex and compact when endowed with the weak-* topology, and furthermore when\footnote{When the $\alpha(K) m \leq 1$, the result is uninteresting, but still true.  $\alpha(K) m =1$ implies $\mathcal{P}_m(K) = \{ m \alpha \}$, while $\alpha(K)m < 1$ implies $\mathcal{P}_m(K)$ is empty.}  $1/\alpha(K)<m<\infty$ the extreme points of $\mathcal{P}_m(K)$ are probability measures $\mu$ of the form
\[
	\mu(A) = \int_A \big( m \mathbbm{1}_U(y) + c \mathbbm{1}_x(y)\big) d\alpha(y).
\]
for $U \subseteq K$, $c \in (0,m)$ and $x \in K \setminus U$.
\end{lem}

\begin{proof} We will assume $m<\infty$, the $m=\infty$ case is straight forward. 
Convexity is immediate. For compactness recall that $\mathcal{P}_m(K) = M^{-1} (\infty,m]$, and thus by lower semi-continuity of $M$ (Lemma \ref{lem:LSC}), $\mathcal{P}_m(K)$ is a closed subset of the compact set $\mathcal{P}(K)$, and hence compact itself.

Now let us show that the proposed extremizers are indeed extreme.  Take $\mu$ of the form $\frac{d \mu}{ d \alpha} = m \mathbbm{1}_U + c \mathbbm{1}_{\{x\}}$, and suppose that $\mu$ is an interpolation $\mu_1$ and $\mu_2$, explicitly $\mu = (1-t) \mu_1 + t \mu_2$.  Then clearly, the support of $\mu_i$ must be contained in the support of $\mu$, and as elements of $\mathcal P_m(K)$, $\mu_i(W) \leq m \alpha(W)$.  If this inequality were strict on some $W \subseteq U$ we would have $m \alpha(W) > (1-t)\mu_1(W) + t \mu_2(W) = \mu(W)$, and we would arrive at a contradiction.  Since $\mu_i$ are probability measures supported on $U \cup \{x\}$, 
\[
	\mu_i(\{x\}) = 1 - \mu_i(U) = 1- m \alpha(U) = \mu(\{x\}).
	\]
Thus $\mu$ is extreme.

Next we will show that extreme points are necessarily of the proposed form.  First for notational brevity, for $z \in [0,m/2)$, let us write $S_z$ for the set $\frac{d \mu}{d \alpha}^{-1}(z, m-z)$, and suppose that there exists $\varepsilon>0$, such that we can find disjoint subsets of $S_\varepsilon$ denoted $U_1$, $U_2$ with $\alpha(U_i) >0$.  For concreteness say $\alpha(U_1) \geq \alpha(U_2)$.  Define elements of $\mathcal P_m(K)$ by
\begin{align*}
	d \mu_1 
		&= 
			\left( \frac{d\mu}{d\alpha} + \varepsilon \left[ \mathbbm 1_{U_1} - \frac{\alpha(U_2)}{\alpha(U_1)} \mathbbm{1}_{U_2} \right] \right) d \alpha
			\\
	d \mu_2
		&=
			\left( \frac{d\mu}{d\alpha} - \varepsilon \left[ \mathbbm 1_{U_1} - \frac{\alpha(U_2)}{\alpha(U_1)} \mathbbm{1}_{U_2} \right] \right) d \alpha
\end{align*}
Then $(\mu_1 + \mu_2)/2 = \mu$, so that such a $\mu$ can never be extreme.

It remains to show that we can find such $\varepsilon,U_1,U_2$ when $\frac{d\mu}{d\alpha}$ cannot be decomposed into the form of our proposed extremizers.  For this purpose it suffices to find $\varepsilon>0$ , and distinct $ x_1, x_2$  belonging to the support of the restricted measure $\mu_{S_\varepsilon}$, (explicitly $\mu_{S_\varepsilon}(A) = \mu(S_\varepsilon \cap A)$).  Indeed take, $\mathcal O_i$ to be disjoint open neighborhoods of  $x_i$ and define 
\[
	U_i = \mathcal O_i \cap S_\varepsilon.
\]

Now, since $\alpha(S_0) >0$ there exists by measure continuity, $\varepsilon'>0$ such that $\alpha(  S_{\varepsilon'} ) >0$, thus $\mu_{S_{\varepsilon'}}$ is not the zero measure and we may choose $x_1$ belonging to its support.  Again by measure continuity and since $\alpha(S_0 \setminus \{x_1\}) >0$ (as otherwise $\mu$ is of the proposed extremal form), so there exists $\varepsilon \in (0, \varepsilon')$ such that $\alpha(S_\varepsilon \setminus \{x\} ) >0$.  Thus choosing our $x_2$ to be an element of the support of $\mu_{S_\varepsilon \setminus \{x_1\} } $, we have the desired $x_i$. 
\end{proof}
Piecing the above together we have the following. 
\vspace{.1in}

\begin{proof}[Proof of Lemma \ref{lem:MainLC}]
Suppose the result holds for $k < n$. Then since coordinate convexity and lower semi-continuity are preserved by the fixing of coordinates,
\begin{align*}
	\Phi(\mu_1, \dots, \mu_n) 
		&\leq 
			\sup_{\nu_n \in \mathcal{E}(\mathcal{P}_{m_n}(K_n))} \Phi(\mu_1, \dots, \mu_{n-1}, \nu_n)
			\\
		&\leq
			\sup_{\nu_n \in \mathcal{E}(\mathcal{P}_{m_n}(K_n))} \left( \sup_{\nu_i \in \mathcal{E}(\mathcal{P}_{m_i}(K_i)), i<n} \Phi(\nu_1, \dots, \nu_{n-1}, \nu_n) \right)
			\\
		&=
			\sup_{\nu_i \in \mathcal{E}(\mathcal{P}_{m_i}(K_i))} \Phi(\nu_1, \dots , \nu_n).
\end{align*}
By Lemma \ref{lem:topo} the $n=1$ case is a direct application of Krein-Milman to $ \mathcal{P}_{m_1}(K_1)$. From this the equations \eqref{inq:MainLC} follow as soon as one observes the general fact that 
\begin{align} \label{eq:pextreme}
\mathcal{E}(S_1 \times \cdots \times S_n) = \mathcal{E}(S_1) \times \cdots \times \mathcal{E}(S_n),
\end{align}
for sets $S_i \subseteq V$ in a vector space.  The $n=2$ case is enough. For $x = (x_1,x_2) \in \mathcal{E}(S_1) \times \mathcal{E}(S_2)$, and suppose that $ x = (u + v)/2 = (\frac{u_1 + v_1}{2}, \frac{u_2 + v_2}{2})$.  Thus $x_i = (u_i + v_i)/2$.  Thus by the extremality of $x_i$, $x_i = u_i = v_i$, and we have $x = u=v$ and $x \in \mathcal{E}(S_1 \times S_2)$.
 
 For $x \in \mathcal{E}(S_1 \times S_2)$, if $x=(x_1,x_2)$ and suppose there exists $u_1,v_1$ such that $x_1 = (u_1 + v_1)/2$ then $u = (u_1,x_2)$ and $v = (v_1, x_2)$ are such that $x = (u+v)/2$ so that $u=v=x$.  As a consequence $x_1 = u_1=v_1$, so that $x_1 \in \mathcal{E}(S_1)$.  It follows similarly that $x_2 \in \mathcal{E}(S_2)$, so that $x \in  \mathcal{E}(S_1) \times \mathcal{E}(S_2)$ and we have \eqref{eq:pextreme}.
 
The characterization of extreme points is given by Lemma \ref{lem:extremes}. 
\end{proof}

{

\subsection{Beyond compactness} 
The following theorem extends Lemma \ref{lem:MainLC} to suit potentially non-compact spaces.    We will investigate further extensions in the presence of a theory of rearrangement, developed in Section \ref{sec:rearrangement} below.
\label{sec:more}

\begin{thm}
	For $\{E_i, \alpha_i \}$ and a coordinate convex, lower semi-continuous (with respect to the product of weak-* topologies),
	$
		\Phi: \prod_{i=1}^n \mathcal{P}_{m_i}(E_i) \to (-\infty, \infty]
	$
	we have
	\begin{align}
		\sup \left\{ \Phi(\mu_1,\ldots,\mu_n):  (\mu_1,\ldots,\mu_n)\in \prod_{i=1}^n \mathcal{P}_{m_i}(E_i) \right\}
		=	\sup  \left\{ \Phi(\mu_1,\ldots,\mu_n):  (\mu_1,\ldots,\mu_n)\in\prod_{i=1}^n \mathcal{P}_{m_i}^*(E_i) \right\}.
	\end{align}
	where we define 
	\[
		\mathcal{P}_{m}^*(E) = 
		\begin{cases}
		\bigcup_{K \subseteq E \mbox{ compact}} \mathcal{E}( \mathcal{P}_{m}(K)) &\text{if $m \alpha (E) >1$}\\
		\mathcal{P}_{m}(E) &\text{otherwise} .
		\end{cases}
	\]
\end{thm}

Let us first remark that $m \alpha (E) >1$ is the only interesting case.  When $m \alpha (E) =1$, $\mathcal{P}_{m}(E)$ consists only of the singleton $m \alpha$, while $m \alpha(E) < 1$ implies that $\mathcal{P}_m(E)$ is empty.
\begin{proof}
	We will prove the claim by induction. The result will hinge on the fact that for a Polish space $E$, with regular measure $\alpha$, the compactly supported measures in $\mathcal{P}_m(E)$ form a dense subset, which we show now.  If $m = \infty$, $\mathcal{P}_m(E) = \mathcal{P}(E)$ and the result is immediate.  Indeed, for $\mu \in \mathcal{P}(E)$ use Prokhorov's tightness theorem to select an increasing sequence of compact sets $K_m$, such that $\lim_m \mu(K_m) =1$, and then take $\mu_m$ defined by $\mu_m(A) = \mu(A \cap K_m)/\mu(K_m)$ to form a sequence of compactly supported probability measures converging to $\mu$.
	
	In the case that $m$ is finite, for $\mu \in \mathcal{P}_m(E)$ we have by Radon-Nikodym, $d\mu = f d \alpha$. Then either there exists an $\varepsilon_0 >0$ such 
	\[
		\alpha(f \leq m - \varepsilon_0) > 0
	\]
	or there does not.  If not, letting $\varepsilon$ tend to zero and using the continuity of measure we have $\alpha( f < m) =0$, but this implies $d \mu = m d \alpha$ and we are done. Any element of $\mathcal{P}_m(E)$ has a density say $g$ bounded above by $m$, so that $f-g \geq 0$.  Thus $\int |f-g| d\alpha = \int f -g d\alpha = 0$, and $\mathcal{P}_m(E)$ is just the singleton $\{\mu \}$.  
	
	In the case there exists such a $\varepsilon_0$, fix such an $\varepsilon_0$ and use inner regularity to choose a compact set $K \subseteq \{f \leq m - \varepsilon_0\}$ such that $\alpha(K) >0$.  Then choose a sequence of increasing compact sets $K_j$ containing $K$ such that $ \lim_j \int_{K_j} f d \alpha = \int f d \alpha$.  Let $f_j$ be a truncation of $f$ in the sense that 
	\[
		f_j(x) =  \begin{cases}
			\min \{f(x), m - \frac 1 j \} &\text{ for } x \in K_j, \\
			0 &\text{ for } x \notin K_j.
		\end{cases}
		\]
		  Then define a density,
	\[
		g_{j} = f_j + \left(\frac{\int f- f_j d\alpha}{\alpha(K)}\right) \mathbbm{1}_K,
	\]
	By definition $g_j$ is compactly supported on $K_j$, non-negative as $f_j \leq f$, and satisfies $\alpha(g_j) = \alpha(f) = 1$.   Thus $\mu_j$ defined by  $\frac{d\mu_j}{d\alpha} = g_j$ gives a sequence of compactly supported measures in $\mathcal{P}(E)$. It remains to show that they approximate $\mu$ in $\mathcal{P}_m(E)$. We first compute that $\|g_j\|_\infty \leq m$ for large $j$ 
	\begin{align*}
		\alpha(g_j > m )
			&=
				\alpha(\{g_j >m \} \cap K^c) + \alpha(\{g_j >m \} \cap K)
					\\
			&=
				\alpha(f_j > m) + \alpha( \{f_j + c_j > m \}\cap K).
	\end{align*}  
		But $\{ f_j + c_j > m\} \cap K$ is only non-empty when $c_j \geq \varepsilon_0$.  By the construction of $K_j$ $c_j$ goes to zero with $j$ large, and since $\alpha(f_j >m) =0$ for all $j$, we have $\alpha(g_j >m) =0$, and thus $\mu_j \in \mathcal{P}_m(E)$ for $j$ large enough.
		
		It remains to check that $\mu_j$ goes to $\mu$ in the weak-* topology.  For $\varphi$ continuous and bounded, a straightforward application of Lebesgue dominated convergence shows that $\mu_j(\varphi) \to \mu(\varphi)$.
		
		By the above, $\cup_{\{K \subset E_i compact\}} \mathcal{P}_{m_i}(K)$ is dense in  $\mathcal{P}_{m_i}(E_i)$ and hence it follows that for a fixed $\varepsilon >0$ and $\mu_i \in \mathcal{P}_{m_i}(E_i)$, there exists compactly supported $\nu_i$ in say $\mathcal{P}_{m_i}(K_i)$ such that (by lower semi-continuity),
		\[
			\Phi(\mu_1, \dots, \mu_n) < \Phi(\nu_1, \dots, \nu_n) + \varepsilon.
		\] 
		But by Lemma \ref{lem:MainLC},
		\begin{align*}
			\Phi(\nu_1, \dots, \nu_n) 
				&\leq
					\sup \Phi( \mathcal{E}( \mathcal{P}_{m_1}(K_1)) , \dots, \mathcal{E}(\mathcal{P}_{m_n}(K_n)))
		\end{align*}
			But by inclusion $\mathcal{E}( \mathcal{P}_{m_i}(K_i)) \subseteq \mathcal{P}_{m_i}^*(E_i)$ and taking $\varepsilon \to 0$ we have our result.
\end{proof}




\section{Rearrangements and Rogozin-type theorems}
\label{sec:rearrangement}

\subsection{A rearrangement theorem for maxima on Euclidean spaces}
\label{sec:rd}

For $X = (X_1, \dots, X_n)$ with $X_i$ independent and distributed on $\mathbb{R}^d$ according to density $f_i$, we consider a 
sort of ``rearrangement" $X_* = (X_1^*, \dots, X_n^*)$ a random variable with independent coordinates $X_i^*$ distributed 
according to  $f_i^*$ the symmetric decreasing rearrangement of $f_i$.  Explicitly, for a non-negative function $f$, 
its symmetric decreasing rearrangement is uniquely defined by the formula,
\[
	f^*(x) = \int_0^\infty \mathbbm{1}_{\{(y,s): f(y)>s\}^*}(x,t) dt,
\]
where $A^*$ denotes the open origin-symmetric ball with the same measure as $A \subseteq \mathbb{R}^d$.
We then show that under linear maps that respect the independence structure of the random variable $X$, 
the essential sup operation is non-decreasing on this rearrangement in the following sense.

\begin{thm}\label{thm:Mainviarearr}
	Suppose that $X = (X_1,\cdots,X_n)$ is comprised of $X_i$ independent random vectors in $\mathbb{R}^d$ and $T$ is a linear map on $\mathbb{R}^d$ to a $k$-dimensional Euclidean space $E^k = T(\mathbb{R}^d)$. Then,
	\be\label{inq:Main1}
	M (T^{(d)} (X))\le M (T^{(d)} (X_*)).
	\ee
\end{thm}

We will employ what is often called the Brascamp-Lieb-Luttinger rearrangement inequality, that 
for any measurable functions $f_i: \mathbb{R}^d\rightarrow[0,\infty)$, with $1 \le i \le n$, and real numbers $a_{ij}, 1 \le i \le n, 1 \le j \le  N$,
\ben
\int_{\mathbb{R}^{dN}}\prod_{j=1}^Ndx_j\prod_{i=1}^Mf_i\left(\sum_{j=1}^Na_{ij}x_j\right)\le \int_{\mathbb{R}^{dN}}\prod_{j=1}^Ndx_j\prod_{i=1}^Mf^*_i\left(\sum_{j=1}^Na_{ij}x_j\right).
\een

We can recast for our notation and interest as the following,
\begin{thm}\label{thm:BLinq}{\cite{Rog57, BLL74}}
For a product density $f(x) = \prod_{i=1}^n f_i(x_i)$, and $A$ an $n \times N$ matrix
\ben
\int_{\mathbb{R}^{dN}} f(A^{(d)} (x)) dx \le \int_{\mathbb{R}^{dN}} f_*(A^{(d)} (x)) dx
\een
where $f_* = \prod_i f_i^*$.
\end{thm}

Let $X = (X_1, \dots, X_n)$ be composed of independent $\mathbb{R}^d$ valued random variables with compact support and density $f$.  Let $T$ be a linear map from $\mathbb{R}^n$ onto a $k$-dimensional subspace of $\mathbb{R}^n$. 
\begin{lem} \label{lem:pushforwardformula}
	If $A=(a_1, \dots, a_{n-k})$ is a matrix composed of column vectors $a_i$ that constitute a basis for $\ker(T)$, then
	\begin{align} \label{eq:density}
		f_{T^{(d)}(X)}(0) 
			&= 
				|TT^t|^{-2/d} \int_{\ker(T)} f(w) dw 
					\\
			&= 
				C(T,A) \int_{(\mathbb{R}^d)^{n-k}} f(A^{(d)} I_d(y)) dy.
	\end{align}
	where $C(T,A) =  |TT^t|^{-2/d}|A^{(d)}|^{-1}$.
\end{lem}
	To be clear, the determinant of $TT^t$ will be considered as an automorphism of the subspace, and $A^{(d)}$ will be considered as a linear isomorphism between $(\mathbb{R}^d)^{n-k}$ and $\ker T^{(d)}$.
\begin{proof}
	By a general change of variables (for example see the co-area formula in \cite{Fed69:book})we can write,
	\[
		f_{T^{(d)}(X)}(0) = |(T^{(d)}) (T^{(d)}I_d)^t|^{-1/2} \int_{\ker(T^{(d)})} f(x) dx
	\]
	Recalling the that $T^{(d)} = T \otimes I_d$, From the general algebraic properties of Kronecker products we have,
	 \begin{align*}
		|(T \otimes I_d)(T \otimes I_d)^t) |
			&=
				|(T \otimes I_d)(T^t \otimes I_d^t)|
					\\
			&=
				|(T T^t\otimes I_d I_d^t)|
					\\
			&=
				| TT^t \otimes I_d |
					\\
			&=
				|TT^t|^d |I_d|^n
					\\
			&=
				|TT^t|^d.
	\end{align*}
	Now we verify that the map $A^{(d)}$ is an isomophism, from $(\mathbb{R}^d)^{n-k}$ to $\ker(T^{(d)})$.  Again using general properties of Kronecker products, $\rank(A ^{(d)}) = \rank(A)\rank(I_d) = (n-k)d$ so to show that $A^{(d)}$ is a linear isomorphism, it is enough to show that it maps to $\ker(T^{(d)})$.  But this is immediate,
	\begin{align*}
		T^{(d)}(A^{(d)}(x)) 
			&=
				(T \otimes I_d)(A \otimes I_d)(x)
				\\
			&=
				TA \otimes I_d I_d (x)
				\\
			&=
				0_n \otimes I_d (x) 
				\\
			&=
				0.	
	\end{align*}
	Our result now follows from change of variables as we will then have,
	\[
		\int_{(\mathbb{R}^d)^{n-k}} f(A^{(d)}(y)) dy = |A^{(d)}| \int_{\ker( T^{(d)})} f(x) dx.
	\]
\end{proof}

We can now prove our result.

\begin{proof}[Proof of Theorem \ref{thm:Mainviarearr}]
	Given an $\varepsilon >0$ we may assume by translation that $f_{T^{(d)}(X)}(0) + \varepsilon > M(T^{(d)}(X))$. By Lemma \ref{lem:pushforwardformula},
	\[
		f_{T^{(d)}(X)}(0) =  C(T,A) \int_{(\mathbb{R}^d)^{n-k}} f(A^{(d)}(y)) dy.
	\]
	Applying Theorem \ref{thm:BLinq},
	\begin{align*}
		 C(T,A) \int_{(\mathbb{R}^d)^{n-k}} f(A^{(d)}(y)) dy
			 &\leq
				  C(T,A)  \int_{(\mathbb{R}^d)^{n-k}} f_*(A^{(d)}(y)) dy
				  \\
			&=
				f_{T^{(d)}(X_*)}(0).
	\end{align*}
	hence,
		\[
			M(T^{(d)} (X)) \leq M(T^{(d)}(X_*)) + \varepsilon,
		\]
		and with $\varepsilon \to 0$, the result holds.
\end{proof}
\subsubsection{A Rogozin-type Equation for linear transforms}
Notice, in the case that $X_i$ are independent and compactly supported, the synergy between Theorems \ref{thm:MainLC} and \ref{thm:Mainviarearr},
\begin{align*}
	M(T^{(d)}(X)) 
		&\leq 
			\sup_U M(T^{(d)}(U))
				\\
		&\leq 
			\sup_U M(T^{(d)} (U_*)).
\end{align*}
But since $U_i$ is uniform on a set such that $M(U_i) = M(X_i)$ its symmetric rearrangement $U_i^*$ is just $Z_i$ uniform on the ball such that $M(Z_i) = M(X_i)$.  We can collect this observation as the following theorem.

\begin{thm}\label{thm:RogozinLinear}
	When $X_i$ are independent $\mathbb{R}^d$ random variables and $T$ is a linear map from $\mathbb{R}^n$ to a $k$-dimensional space,
	\begin{align*}
			M(T^{(d)}(X)) \leq M(T^{(d)} (Z)).
	\end{align*}
	Where $Z$ is composed of independent $Z_i$, distributed uniformly on the origin symmetric ball such that $M(Z_i) = M(X_i)$.
\end{thm}

\subsection{The integers}
The following is a rather special case of a more general rearrangement theory for integer valued random variables.

\begin{lem} \cite{WWM14:isit, MWW17:1, MWW17:2} \label{thm:discrearrange}
	When $\{U_i\}_{i=1}^n$ are uniformly distributed on subsets of size $l_i$, then 
	\begin{equation*}
	M(U_1 + \cdots + U_n) \leq M(Z_1 + \cdots + Z_n)
	\end{equation*}
	where $Z_i$ is uniformly distributed on $\{0,1,\dots, l_i-1\}$.
\end{lem}

Combining our main theorem, with the above we have the following
\begin{thm}
	When $X_1, \cdots, X_n$ are independent integer valued random variables satisfying $M(X_i) \leq 1/l_i$ for integers $l_i$, where the maximum is considered with respect to the usual counting measure, then 
	\begin{equation*}
	M(X_1 + \cdots + X_n) \leq M(Z_1 + \cdots + Z_n)
	\end{equation*} 
	where $Z_i$ is uniformly distributed on $\{0,1, \dots l_i-1\}$.
\end{thm}
\begin{proof}
	Apply Theorem \ref{thm:MainLC} and then Lemma \ref{thm:discrearrange}.
\end{proof}




\section{Quantitative bounds based on Geometric inequalities}
\label{sec:Quantitative bounds}


\label{ss:epi}
\subsection{Euclidean Space}

In Brzezinski \cite{Brz13}, the author proves a geometric inequality about certain sections of $n$-products of $d$-dimensional balls.
Namely that for $\theta \in S^{n-1}$, there exists an explicit $c(d)>0$ (for which we will give further discussion) such that, 
\ben
\left|\left\{ x \in \prod_{i=1}^n B(0,r_d):\sum_{j=1}^n \theta_jx_j=0\right\} \right|_{d(n-1)} \leq \hspace{3mm} c(d).
\een


Letting $\{Z_i\}_{i=1}^n$ independent random vectors uniformly distributed on the ball of unit volume in $\mathbb{R}^d$, we can describe this result probabilistically as the fact that for $\theta \in S^{n-1}$
\ben
M(\theta_1 Z_1 + \cdots + \theta_n Z_n)\le  c(d).
\een

A projection matrix from $\mathbb{R}^n$ to a one dimensional subspace, can be realized by taking a representative unit vector $\theta$, and writing $P_\theta x = (\theta_1 x_1 + \cdots + \theta_n x_n) \theta$.  Thus $P_\theta \otimes I_d(Z) = \theta \otimes I_d (\theta_1 Z_1 + \cdots + \theta_n Z_n)$, interpreting $\theta$ here as a linear functional.  But $\theta \otimes I_d$ is a linear isometry from $\mathbb{R}^d$ onto its image.  Indeed, 
\begin{align*}
\|\theta \otimes Id (x) \|_{nd}^2
&=
\| (\theta_1 x, \dots, \theta_n x) \|^2
\\
&=
\sum_i \sum_j (\theta_i x_j)^2
\\
&=
\sum_i \theta_i^2 \|x\|^2_{d}
\\
&=
\|x\|_d^2.
\end{align*}
Thus,
\[
M(\theta_1 Z_1 + \cdots + \theta_n Z_n) = M(P_\theta \otimes I_d(Z)) \leq c(d).
\]
.

 Before we state our theorem let us remark on the constant $c(d)$.  In \cite{Bal86}, K. Ball proved $c(1)= \sqrt{2} $ (which in turn is  $c(1,1)$ in our Theorem~\ref{thm:Generalized infinity EPI}) as the best possible bound by deriving a sharp control of the $L^p$ norm of the Fourier transform of an interval indicator. In fact soon after in \cite{Bal89} Ball obtained Theorem~\ref{thm:Generalized infinity EPI} in the special case that each $X_i$ is distributed uniformly on $[-\frac 1 2, \frac 1 2]$.  In \cite{Brz13} Brzezinski used a similar approach to derive a best possible bound for $c(d)$ when $d \geq 2$, dependent on precise control of the $L^p$ norm of the Fourier transform of a the uniform distribution on the $d$-ball of unit volume, (which turn out to be normalized Bessel functions).  What is more, using estimates on  certain Bessel functions, Brzeziski showed $c(d)$ to be asymptotically sharp by showing the bound obtained for a certain sequence of vectors.
 
 This result can be reformulated probabilistically as 
\[
c(d) = \lim_n M \left( \frac{Z_1 + \cdots +Z_n}{\sqrt n}\right).
\]
Where $Z_i$ are iid and uniformly distributed on the unit volume ball. With this probabilistic formulation, there's an obvious and pleasant approach to obtaining the value of the right hand side.  Using log-concavity with symmetry to locate the maximum of the density and the local central limit theorem (see \cite{Dur10:book} for example),
\begin{align*}
c(d) = M(Z_G)
\end{align*}
where $Z_G$ is normal with covariance matrix $\Sigma$ matching $B_1$.  Hence, 
\begin{align} \label{eq:c(d)}
c(d) = (2 \pi)^{-d/2}|\Sigma|^{-1/2}.
\end{align}
The covariance matrix of a random variable uniformly distributed on a $d$-ball of radius $R$, is $\frac{R^2}{d+2} I_d$, and the radius of a $d$-ball of unit volume is $\Gamma^{\frac 1 d}(1 + \frac d 2) \pi^{-1/2}$.  Combining these we have
\begin{align*}
c(d) 
	 &= 
		c(d,1)
			\\
	&=
		\frac{(1+ \frac d 2)^{\frac{d}{2}}}{\Gamma(1+\frac{d}{2})}.
\end{align*}
Let us record this in the following
\begin{thm}\label{thm:brz}  {\cite{Brz13}}
	For $Z = (Z_1, \dots, Z_n)$, with coordinates independent and uniformly distributed on the $d$-ball of unit volume, and a one dimensional projection matrix $P$ on $\mathbb{R}^n$,
	\begin{align}
	M(P^{(d)}(Z)) \leq \frac{(1 + \frac d 2)^{\frac{d}{2}}}{\Gamma(1+\frac{d}{2})}.
	\end{align}
\end{thm}

The proof of the above hinges on the following integral approximation, which will be useful in the generalization put forth below.

 \begin{lem}\cite{Brz13} \label{lem:brzest}
	If $\phi_d(\xi)$ denotes the characteristic function of a random variable uniformly distributed on the $d$-ball of unit volume and $p \geq 2$,
	\begin{align}
	(2 \pi)^{-d} \int_{\mathbb{R}^d} | \phi_d(\xi)|^p d \xi 
	\leq 
	c(d) p^{-\frac d 2}.
	\end{align}
	More generally for $Z$ uniformly distributed on a $d$ ball,
	\begin{align}
		(2 \pi)^{-d} \int_{\mathbb{R}^d} | \phi_Z(\xi)|^p d \xi 
			\leq 
				M(Z) c(d) p^{-\frac d 2}
	\end{align}
\end{lem}

Before we proceed we will also need use of a geometric Brascamp-Lieb inequality, and an elementary lemma.
\begin{thm}\cite{BH09,BCCT08} \label{thm:GBLI}
	For $i=1,\cdots m$, $c_i >0$ and $A_i : H \to H_{i}$ be linear maps from a Euclidean space of dimension $N$ to a subspace of dimension $n_i$ such that $A_i A_i^t = I_{n_i}$ with 
	\begin{align*}
	\sum_{i=1}^m c_i A_i^t A_i = I_N,
	\end{align*}
	then for all Borel functions $g_i : H_i \to \mathbb{R}^+$,
	\begin{align*}
	\int_{H} \prod_{i=1}^m g_i(A_i x)^{c_i} dx
	\leq
	\prod_{i=1}^m \left( \int_{H_i} g_i \right)^{c_i}.
	\end{align*}
\end{thm}

\begin{lem}\label{lem:LPPlem}
	Let $\theta=(\theta_1,\cdots, \theta_n)\in S^{n-1}$, denote by $P_i:\mathbb{R}^n\rightarrow\mathbb{R}^n$ the orthogonal projection onto $e_i^{\perp}\subset\mathbb{R}^n$. Then for $d\ge 1$, consider $(\mathbb{R}^d)^n$, let $A$ be a measurable subset of $(\langle \theta\rangle^\perp)^d$ with $\mbox{span}(A)=M^d$ for some subspace $M\subset\mathbb{R}^n$ with $\mbox{dim}(M)=m\in \{1,\cdots, n-1\}$. Then we have
	\ben
	|\theta_i|^d|A|_{md}\le |P_i^{(d)}(A)|_{md}
	\een
\end{lem}

We will only sketch a proof, the mechanics of the  argument are the same as Lemma 3.1 of \cite{LPP15} where more details can be found.

\begin{proof} We may assume that $P_i:\theta^\perp\rightarrow e_i^{\perp}$ is injective and $\theta \neq \pm e_i$. Standard linear algebraic techniques can obtain an orthonormal basis for $M$ $\{v_1,\cdots, v_m\}$ with $|P_i(v_1)|=|\theta_i|$ and $P_i(v_j)=v_j$ for all $j=2,\cdots, k$. Now we will raise from $\mathbb{R}^n$ to $(\mathbb{R}^d)^n$. Define $v_{j,l}:=v_j\otimes\delta_d^l$ where $\delta_d^{l}$ denotes the $d$-dimensional vector with the $l$-th position equals to 1 and other positions equal to 0. It is easy to verify that $|P_i^{(d)}(v_{1,l})|=|\theta_i|$ for all $l=1,\cdots, d$, and $P_i^{(d)}(v_{j,l})=v_{j,l}$ for all $j=2,\cdots, k$ and $l=1, \cdots, d$. hus a similar argument as Lemma 3.1 in \cite{LPP15} provides this lemma. 
\end{proof}


Now let us state the main result of this section.
\begin{thm}\label{thm:genlpp}
	Given $Z = (Z_1, \cdots, Z_n)$ comprised of independent, random vectors such that each $Z_i$ is distributed uniformly on a ball in $\mathbb{R}^d$ with $d\geq 2$, and a projection matrix $P$ from $\mathbb{R}^n$ to a $k$ dimensional subspace, then there exists $\{\gamma_i\}_{i=1}^n$, with $\gamma_i \in [0,1]$ such that $\sum_{i=1}^n \gamma_i =k$ such that
	\be\label{inq:GeneralizedLPP}
	M(P^{(d)} (Z)) \le c(d,k)\prod_{i=1}^n M(Z_i)^{\gamma_i},
	\ee
	where $c(d,k)$ is the minimum of two separately derived constants,
	\be
		  c_1(d,k) = \left(\frac{(1 + \frac d 2 )^{\frac{d}{2}}}{\Gamma(1 + \frac d 2)} \right)^k 
			  \hspace{5mm} 
		c_2(d,k) = \left(\frac{n} {n-k} \right)^{\frac{d(n-k)}{2}}.
	\ee
\end{thm} 

The assumption that $d\ge 2$, is at no loss of generality, the case $d=1$ is solved by \cite{Bal89, LPP15} with 
\[
	c(1,k) = \min \left\{2^{\frac k 2}, \left(\frac{n}{n-k} \right)^{\frac{n-k}{2}}\right\}
	\]

 Let us notice that since $Z = M(Z) \mathbbm{1}_B$, with $B = (B_1, \dots, B_n)$, this statement has an equivalent, more obviously geometric formulation, that for any set of $B_i$ $d$-dimensional balls, with their product $B = \prod_{i=1}^n B_i$, there exist $\beta_i \in [0,1]$ such that $\beta_1 +\cdots + \beta_n = n-k$, such that
\begin{align}
|\{E^{(d)}_k \cap B\}|_{d(n-k)} \leq c_1(d,k) \prod_{i=1}^n |B_i|_d^{\beta_i},
\end{align}	
where $E_k$ is a $k$ co-dimensional subspace, and
	\begin{align}
	\label{def:spacefolding}
		E^{(d)} = E \otimes \mathbb{R}^d \subseteq (\mathbb{R}^d)^n  = \mathbb{R}^n \otimes \mathbb{R}^d
	\end{align}  with the orthogonal projection matrix given by $ P^{(d)}$ mapping from $(\mathbb{R}^d)^n\rightarrow E^d$.

\begin{proof}
	We will prove the result in two parts, as it is enough to find $\gamma_i$ for $c_1$ and then $c_2$.
	
\noindent{\bf Proof of the constant $\left(\frac{(\frac d 2 +1)^{\frac{d}{2}}}{\Gamma\left(\frac{d}{2} +1 \right)}\right)^{k}$.} 

Assuming the result inductively for the case of $n-1$,   we may assume that $P e_i$ is non-zero for all $i$.  In this case we have the following decomposition,
\[
P = PP^T = P I_n P = P \left( \sum_i e_i e^T_i \right) P^T = \sum_i Pe_i(Pe_i)^T.
\]
Writing $a_i = |Pe_i|$ and $u_i = Pe_i/|Pe_i|$ we have,
\be\label{eq:decomP}
P = \sum_{i=1}^n a_i^2 u_i u_i^T.
\ee
From this it follows that $P^{(d)} = \sum_{i=1}^n a_i^2 (u_iu_i^T)^{(d)}= \sum_{i=1}^n a_i^2 \left(u_i^{(d)}(u_i^{(d)})^T\right)$.

Again from the inductive hypothesis, we may assume that all $a_i>0$. We argue in two separate cases, the first will be the case in which $E_k$ lies orthogonal to a unit vector $\theta$ with some coordinate $|\theta_i|>1/\sqrt{2}$ where we will use the geometric formulation above, and the second case in which it does not where we will employ the original formulation.

We proceed in the case that $E_k \subseteq \theta^\perp$ with $\theta$ possessing a ``dominant coordinate", and show that the problem can be projected down so that induction may be applied.  Indeed, take $\theta$ to be perpendicular to $E_k$, with a coordinate $\theta_1$ (without loss of generality) greater than $2^{-1/2}$, then 
if we write $T_i$ as the orthogonal projection to $e_i^\perp$, then by Lemma \ref{lem:LPPlem} for any $\theta_i$ the $i$-th coordinate in a unit vector $\theta$ such that $E_k \subseteq\theta^\perp$
	 \begin{align*}
		 |\{E^{(d)}_{k} \cap B\}|_{dk} \leq |\theta_i|^{-d} |T_i\otimes I_d(E^{(d)}_k \cap B)|_{d(n-k)}
	 \end{align*}
	 
	 \begin{align*}
		 |\{E^{(d)}_k \cap B\}|_{d(n-k)}
			 &\leq
				 \left|E^{(d)}_k \cap \left( \mathbb{R}^d \times \prod_{i=2}^n B_i \right)\right|_{d(n-k)}
				 	\\
			 &\leq
				2^{\frac d 2} \left|T_1 \otimes I_d \left(E^{(d)}_k \cap \left( \mathbb{R}^d \times \prod_{i=2}^n B_i \right)\right)\right|_{d(n-k)}
					\\
			&\leq 
				2^{\frac d 2} \left| (T_1(E_k) \otimes \mathbb{R}^d ) \cap \prod_{i=2}^n B_i \right|_{d(n-k)}.
	 \end{align*} 
	But now we can apply the induction hypothesis, as $T_1(E_k)$
		\begin{align*}
		\left|T_1(E_k) \otimes \mathbb{R}^d  \cap \prod_{i=2}^n B_i \right|_{d(n-k)} 
			&\leq 
				 c(d,k-1) \prod_{i=2}^n |B_i|^{\beta_i},
	\end{align*}
	where $\sum_i \beta_i = n-k$.  Stirling's formula implies that $2^{\frac d 2} c(d,k-1) \leq c_1(d,k-1)$, and $|B_1|^{\beta_1} =1 $ if $\beta_1 =0$. Combining these results we see that 
	\begin{align*}
		|\{E_k^{(d)} \cap B\}|
			&\leq 
				2^{\frac d 2} c_2(d,k-1) 1 \prod_{i=2}^n |B_i|^{\beta_i}
				\\
			&\leq
				 c_2(d,k)  \prod_{i=1}^n |B_i|^{\beta_i}.
	\end{align*}
	Thus we can restrict ourselves to the case that vectors orthogonal to $E_k$ do not have a ``dominant coordinate".  In this case we proceed via a Fourier theoretic argument.
 Recall the definition of $a_j$ and $u_j$ and by \eqref{eq:decomP}, since $a_i$ is the $j^{th}$ coordinate of $u_j$, so we have all $a_j^2 \in (0, 1/2]$. The standard Fourier inversion formula gives us
	\begin{align*}
		M(P^{(d)}(Z)) 
			\leq
				\frac{1}{(2 \pi)^{kd}} \int_{\mathbb{E}_k^{(d)} } |\phi_{P^{(d)}(Z)}(w)| dw.
	\end{align*}
	As is to be expected, the independence of the $Z_i$ can be exploited.  Considering  $e_i\in\mathbb{R}^n$ as a linear map $\lambda \mapsto \lambda e_i$ matrix, we can express 
	\ben
	Z=\sum_{i=1}^n e_i^{(d)} Z_i.
	\een
	If we further treat $w$ as an $nd\times 1$ vector,
	\begin{align*}
		\langle w , P^{(d)} Z \rangle 
			&=
				\sum_{j=1}^n \langle w, P^{(d)} e_j^{(d)} Z_j \rangle
				\\
			&=
				\sum_{j=1}^n \langle w, a_i u_i^{(d)} Z_j \rangle
				\\
			&=
				\sum_{j=1}^n \langle a_j (u_j^{(d)})^t w, Z_j \rangle. 
	\end{align*} then by independence of $Z_i$,
	\begin{align*}
	\phi_{P^{(d)}(Z)}(w) 
		&=
			\mathbb{E}\exp i \left(\sum_{j=1}^n \langle a_j (u_j^{(d)})^t w, Z_j \rangle \right)
				\\
		&=
			\prod_{i=1}^n \mathbb{E} \exp i  \langle a_j (u_j^{(d)})^t w, Z_j \rangle
				\\
		&=
			\prod_{i=1}^n \phi_{Z_j}(a_j (u_j^t)^{(d)} w)
	\end{align*}
	
	Then applying Theorem \ref{thm:GBLI}, to $A_j = (u_j^{(d)})^t$, $g_j(x) = |\phi_{Z_j}(a_jx)|$, $c_j = a_j^2$, so that by the above we have
	\begin{align*}
	M(P^{(d)}(Z))
		&\leq
			(2\pi)^{-kd} \int_{E^{(d)}} \prod_{j=1}^n g_j (A_j w) dw
				\\
		&\leq 
			(2\pi)^{-kd}  \prod_{j=1}^n \left( \int g_j^{ {a^{-2}_j}} (x) dx \right)^{a^2_j}
	\end{align*}
The application of the geometric Brascamp-Lieb is legitimate by the decomposition of the projection to $E_k^{(d)}$, $P^{(d)} = \sum_j a_j^2 u_j^{(d)} (u_j^{(d)})^T$.  Note, $P$ is the identity on $E_k$ so $P^{(d)}$ acts as the identity on $E_k^{(d)}$ by definition and that $u_j$ is a unit vector in $E_k$ implies $(u_j^{(d)})^T u_j^{(d)} = I_d$.
	
	 We can now apply Lemma \ref{lem:brzest}, reminding the reader that $\sum_j a_j^2 = k$.
	\begin{align*}
	(2\pi)^{-kd}  \prod_{j=1}^n \left( \int g_j^{ {a^{-2}_j}} (x) dx \right)^{a^2_j}
	&=
	{(2\pi)^{-kd}} \prod_{j=1}^n \left( \int_{\mathbb{R}^d}|\phi_{Z_j} (a_j z)|^{\frac{1}{a_j^2}}dz\right)^{a_j^2}
	\\
	&=
	\prod_{j=1}^n \left( (2\pi a_j)^{-d} \int_{\mathbb{R}^d}|\phi_{Z_j} ( z)|^{\frac{1}{a_j^2}}dz\right)^{a_j^2}	
	\\
	&\leq
	\prod_{j=1}^n \left( M(Z_j) c(d) a_j^d a_j^{-d} \right)^{a_j^2}	
	\\
	&=
	c^k(d) \prod_{j=1}^n M^{a_j^2}(Z_j)   	
	\end{align*}
	
	Taking $\beta_j:=a_j^2$ completes the first half the proof.

	\noindent{\bf Proof of the constant $\left(\frac{n}{n-k}\right)^{\frac{d(n-k)}{2}}$.}
	 
	Recall that by symmetry, logconcavity, and the representation given in Lemma \ref{lem:pushforwardformula}
	\begin{align*}
	M(P^{(d)}(Z))
	&= |B|^{-1} \int_{(E^d)^\perp}\prod_{i=1}^n  \mathbbm{1}_{B_i}(y)dy
	\end{align*}
	We will again pursue a Brascamp-Lieb argument.  Denote $\tilde{P}$ the orthogonal projection from $\mathbb{R}^n$ onto $E^\perp$, so that $\tilde{P}^{(d)}$ is the orthogonal projection from $(\mathbb{R}^d)^n$ onto $E^d$. Let $c_i:=|\tilde{P}e_i|$, let $\tilde{u}_i:=\tilde{P}e_i/|\tilde{P}e_i|$, then for $y \in (E^d)^\perp$ it follows from $P^{(d)}y = y$, and $P$ symmetric that, $\mathbbm{1}_{B_i}(y) = \mathbbm{1}_{\frac 1 {c_i} B_i}((u_i^{(d)})^T y)$ $\tilde{P}$.  Applying Theorem \ref{thm:GBLI} to $A_i = (u_i^{(d)})^T$ and $g_i = \mathbbm{1}_{B_i/c_i}$,
	We have

	\begin{align*}
	M( P^{(d)}(Z))
		&= 
			\prod_{i=1}^n M(Z_i) \int_{(E^d)^\perp}\prod_{i=1}^n  \mathbbm{1}^{c_i^2}_{(c_i^{-1} B_i)}((u_i^{(d)})^T y)dy
				 \\
		&\leq 
			\prod_{i=1}^n M(Z_i) \left( \int \mathbbm{1}_{(c_i^{-1} B_i)} \right)^{c_i^2}
				\\
		&=
			\prod_{i=1}^n M(Z_i)^{1-c_i^2} c_i^{-d c_i^2}
	\end{align*}
	Note that $\sum_{i=1}^nc_i^2=n-k$, and that the product $\prod_{j=1}^n\left(c_j^2\right)^{c_j^2}$ is minimized when $c_i^2$'s are equal. Thus we have our conclusion
	\begin{align*}
	M(P^{(d)}(Z))
	\le \left(\frac{n}{n-k}\right)^{\frac{d(n-k)}{2}}\prod_{i=1}^n M(Z_i)^{1-c_i^2}.
	\end{align*}
	And note that if we denote $\gamma_i:=1-c_i^2$ then $\sum_{i=1}^n\gamma_i=k$, this provides the second constant. 
\end{proof}

Combining, Theorems \ref{thm:GRGZ} and \ref{thm:genlpp} we have Theorem \ref{thm:Generalized infinity EPI}.  Let us also state the theorem's information theorec formulation.

\begin{thm}\label{thm:genlppppppp}
	Letting $m$ denote the dimension of the kernel of a projection matrix $P$, one can attribute to $P$ a collection of $t_i \in [0, \frac 1 k]$ such that $\sum_{i=1}^n t_i =1$, and   
	\be\label{inq:GeneralizedLPP1}
N_\infty (P^{(d)}(X)) \ge \max \left( \left(\frac m n \right)^{\frac m k }, \frac{\Gamma^{\frac 2 d}( 1 + \frac d 2)}{1+ \frac d 2} \right) \prod_{i=1}^n N_\infty^{t_i}(X_i)
	\ee
	For any random variable $X = (X_1, \dots, X_n)$ composed of $X_i$ independent $\mathbb{R}^d$ vectors.
\end{thm} 

This formulation is an immediate consequence of raising 
	\begin{align} \label{eq:IEPI}
	M(P^{(d)} X) \leq c(d,k) \prod_{i=1}^n M^{\gamma_i}(X_i)
	\end{align}
	 to the $-2/dk$ power and recalling definitions.  

Notice the pleasant dichotomy that emerges in the information theoretic formulation. The term $\left(\frac m n \right)^{\frac m k }$ dependents only on the dimension of the projection and is independent of $d$, while $\frac{\Gamma^{\frac 2 d}( 1 + \frac d 2)}{1+ \frac d 2}$ is determined complete by $d$ and is independent of dimension of the projection.


\subsection{The Integers}
The main result of Mattner and Roos \cite{MR08:a}, was the following
\begin{thm}\cite{MR08:a} \label{thm:MR08}
	When $\{Z_i\}_{i=1}^n$ are $\mathbb{Z}$ valued random variables, uniformly distributed on $\{0,1, \dots, l-1\}$, then for $n>2$
	\[
	M(Z_1 + \cdots + Z_n) < \sqrt{\frac{6}{\pi(l^2-1)n}}.
	\] 
\end{thm}
Geometrically this is a precise upper bound on the number of a points in the discrete ``$n$-box" that can occupy a plane orthogonal to the $(1,\dots,1)$ vector.
We have the following immediate extension. 
\begin{thm}
	For $X_i$ integer valued random variables satisfying $M(X_i) \leq \frac 1 k$ for an integer $k$ when considered with respect to the usual counting measure, then for $n>2$,
	\[
	M(X_1 + \cdots + X_n) < \sqrt{\frac{6}{\pi(k^2-1)n}}
	\]
\end{thm}
\begin{proof}
	By \ref{thm:MainLC}, $M(X_1 + \cdots + X_n) \leq \sup_U M(U_1 + \cdots + U_n)$.  Applying \ref{thm:discrearrange} we have, $M(U_1 + \cdots + U_n) \leq M(Z_1 + \cdots + Z_n),$ from which our result follows from \ref{thm:MR08}.
\end{proof}

\label{sec:conc}

\nocite{MS73:2}

\bibliographystyle{plain}
\bibliography{$HOME/Dropbox/MAIN/WRITINGS/CommonResources/pustak}

\begin{thebibliography}{10}

\bibitem{ABBN04:1}
S.~Artstein, K.~M. Ball, F.~Barthe, and A.~Naor.
\newblock Solution of {S}hannon's problem on the monotonicity of entropy.
\newblock {\em J. Amer. Math. Soc.}, 17(4):975--982 (electronic), 2004.

\bibitem{Bal86}
K.~Ball.
\newblock Cube slicing in {${\bf R}^n$}.
\newblock {\em Proc. Amer. Math. Soc.}, 97(3):465--473, 1986.

\bibitem{Bal89}
K.~Ball.
\newblock Volumes of sections of cubes and related problems.
\newblock In {\em Geometric aspects of functional analysis (1987--88)}, volume
  1376 of {\em Lecture Notes in Math.}, pages 251--260. Springer, Berlin, 1989.

\bibitem{BN96:unpub}
K.~Ball and F.~Nazarov.
\newblock {Little Level Theorem and Zero-Khinchin Inequality}.
\newblock {\em Unpublished}, 1996.

\bibitem{BJKM10}
A.~D. Barbour, O.~Johnson, I.~Kontoyiannis, and M.~Madiman.
\newblock Compound {P}oisson approximation via information functionals.
\newblock {\em Electron. J. Probab.}, 15(42):1344--1368, 2010.

\bibitem{Bar86}
A.R. Barron.
\newblock Entropy and the central limit theorem.
\newblock {\em Ann. Probab.}, 14:336--342, 1986.

\bibitem{BH09}
F.~Barthe and N.~Huet.
\newblock On {G}aussian {B}runn--{M}inkowski inequalities.
\newblock {\em Studia Math.}, 191(3):283--304, 2009.

\bibitem{BCCT08}
J.~Bennett, A.~Carbery, M.~Christ, and T.~Tao.
\newblock The {B}rascamp-{L}ieb inequalities: finiteness, structure and
  extremals.
\newblock {\em Geom. Funct. Anal.}, 17(5):1343--1415, 2008.

\bibitem{Ber74}
P.~Bergmans.
\newblock A simple converse for broadcast channels with additive white gaussian
  noise.
\newblock {\em IEEE Trans. Inform. Theory}, 20(2):279--280, 1974.

\bibitem{Bil99:book}
P.~Billingsley.
\newblock {\em Convergence of Probability Measures}.
\newblock John Wiley {\&} Sons Inc., New York, second edition, 1999.

\bibitem{BM11:cras}
S.~Bobkov and M.~Madiman.
\newblock Dimensional behaviour of entropy and information.
\newblock {\em C. R. Acad. Sci. Paris S\'er. I Math.}, 349:201--204, F\'evrier
  2011.

\bibitem{BC14}
S.~G. Bobkov and G.~P. Chistyakov.
\newblock Bounds for the maximum of the density of the sum of independent
  random variables.
\newblock {\em Zap. Nauchn. Sem. S.-Peterburg. Otdel. Mat. Inst. Steklov.
  (POMI)}, 408(Veroyatnost i Statistika. 18):62--73, 324, 2012.

\bibitem{BM13:goetze}
S.~G. Bobkov and M.~M. Madiman.
\newblock On the problem of reversibility of the entropy power inequality.
\newblock In {\em Limit theorems in probability, statistics and number theory},
  volume~42 of {\em Springer Proc. Math. Stat.}, pages 61--74. Springer,
  Heidelberg, 2013.
\newblock Available online at {\tt arXiv:1111.6807}.

\bibitem{BLL74}
H.~J. Brascamp, E.~H. Lieb, and J.~M. Luttinger.
\newblock A general rearrangement inequality for multiple integrals.
\newblock {\em J. Functional Analysis}, 17:227--237, 1974.

\bibitem{Brz13}
P.~Brzezinski.
\newblock Volume estimates for sections of certain convex bodies.
\newblock {\em Math. Nachr.}, 286(17-18):1726--1743, 2013.

\bibitem{Cos85a}
M.H.M. Costa.
\newblock On the {G}aussian interference channel.
\newblock {\em IEEE Trans. Inform. Theory}, 31(5):607--615, 1985.

\bibitem{Dur10:book}
R.~Durrett.
\newblock {\em Probability: theory and examples}.
\newblock Cambridge Series in Statistical and Probabilistic Mathematics.
  Cambridge University Press, Cambridge, fourth edition, 2010.

\bibitem{EGZ15}
Yu.~S. Eliseeva, F.~G\"otze, and A.~Yu. Zaitsev.
\newblock {Arak inequalities for concentration functions and the
  Littlewood--Offord problem}.
\newblock {\em Preprint, {\tt arXiv:1506.09034}}, 2015.

\bibitem{Fed69:book}
H.~Federer.
\newblock {\em Geometric measure theory}.
\newblock Die Grundlehren der mathematischen Wissenschaften, Band 153.
  Springer-Verlag New York Inc., New York, 1969.

\bibitem{FMW16}
M.~Fradelizi, M.~Madiman, and L.~Wang.
\newblock Optimal concentration of information content for log-concave
  densities.
\newblock In C.~Houdr\'{e}, D.~Mason, P.~Reynaud-Bouret, and J.~Rosinski,
  editors, {\em High Dimensional Probability VII: The Carg\`ese Volume},
  Progress in Probability. Birkh\"auser, Basel, 2016.
\newblock Available online at {\tt arXiv:1508.04093}.

\bibitem{Hoc14:1}
M.~Hochman.
\newblock On self-similar sets with overlaps and inverse theorems for entropy.
\newblock {\em Ann. of Math. (2)}, 180(2):773--822, 2014.

\bibitem{JKM13}
O.~Johnson, I.~Kontoyiannis, and M.~Madiman.
\newblock Log-concavity, ultra-log-concavity, and a maximum entropy property of
  discrete compound {Poisson} measures.
\newblock {\em Discrete Appl. Math.}, 161:1232--1250, 2013.
\newblock DOI: 10.1016/j.dam.2011.08.025.

\bibitem{JL15}
T.~Ju{\v{s}}kevi{\v{c}}ius and J.~D. Lee.
\newblock Small ball probabilities, maximum density and rearrangements.
\newblock {\em Preprint, {\tt arXiv:1503.09190}}, 2015.

\bibitem{Kan76:1}
M.~Kanter.
\newblock Probability inequalities for convex sets and multidimensional
  concentration functions.
\newblock {\em J. Multivariate Anal.}, 6(2):222--236, 1976.

\bibitem{Kan77}
M.~Kanter.
\newblock Unimodality and dominance for symmetric random vectors.
\newblock {\em Trans. Amer. Math. Soc.}, 229:65--85, 1977.

\bibitem{KM40}
M.~Krein and D.~Milman.
\newblock On extreme points of regular convex sets.
\newblock {\em Studia Math.}, 9:133--138, 1940.

\bibitem{LS01}
W.~V. Li and Q.-M. Shao.
\newblock Gaussian processes: inequalities, small ball probabilities and
  applications.
\newblock In {\em Stochastic processes: theory and methods}, volume~19 of {\em
  Handbook of Statist.}, pages 533--597. North-Holland, Amsterdam, 2001.

\bibitem{LPP15}
G.~Livshyts, G.~Paouris, and P.~Pivovarov.
\newblock On sharp bounds for marginal densities of product measures.
\newblock {\em Preprint, {\tt arXiv:1507.07949}}, 2015.

\bibitem{LYZ04}
E.~Lutwak, D.~Yang, and G.~Zhang.
\newblock Moment-entropy inequalities.
\newblock {\em Ann. Probab.}, 32(1B):757--774, 2004.

\bibitem{MS73:2}
S.~O. Macdonald and A.~P. Street.
\newblock On {C}onway's conjecture for integer sets.
\newblock {\em Bull. Austral. Math. Soc.}, 8:355--358, 1973.

\bibitem{MK17}
M.~Madiman and I.~Kontoyiannis.
\newblock {Entropy bounds on abelian groups and the Ruzsa divergence}.
\newblock {\em Preprint, {\tt arXiv:1508.04089}}, 2015.

\bibitem{MMT12}
M.~Madiman, A.~Marcus, and P.~Tetali.
\newblock Entropy and set cardinality inequalities for partition-determined
  functions.
\newblock {\em Random Struct. Alg.}, 40:399--424, 2012.

\bibitem{MMX17:0}
M.~Madiman, J.~Melbourne, and P.~Xu.
\newblock Forward and reverse entropy power inequalities in convex geometry.
\newblock In E.~A. Carlen, M.~Madiman, and E.~Werner, editors, {\em Convexity
  and Concentration}, IMA Volumes in Mathematics and its Applications.
  Springer, to appear.
\newblock Available online at {\tt arXiv:1604.04225}.

\bibitem{MWW17:2}
M.~Madiman, L.~Wang, and J.~O. Woo.
\newblock {Discrete entropy power inequalities via Sperner theory}.
\newblock {\em Preprint}, 2017.

\bibitem{MWW17:1}
M.~Madiman, L.~Wang, and J.~O. Woo.
\newblock Entropy inequalities for sums in prime cyclic groups.
\newblock {\em Preprint}, 2017.

\bibitem{MR08:a}
L.~Mattner and B.~Roos.
\newblock Maximal probabilities of convolution powers of discrete uniform
  distributions.
\newblock {\em Statist. Probab. Lett.}, 78(17):2992--2996, 2008.

\bibitem{NV13}
H.~H. Nguyen and V.~H. Vu.
\newblock Small ball probability, inverse theorems, and applications.
\newblock In {\em Erd\"os centennial}, volume~25 of {\em Bolyai Soc. Math.
  Stud.}, pages 409--463. J\'anos Bolyai Math. Soc., Budapest, 2013.

\bibitem{Rog57}
C.~A. Rogers.
\newblock A single integral inequality.
\newblock {\em J. London Math. Soc.}, 32:102--108, 1957.

\bibitem{Rog87:1}
B.~A. Rogozin.
\newblock An estimate for the maximum of the convolution of bounded densities.
\newblock {\em Teor. Veroyatnost. i Primenen.}, 32(1):53--61, 1987.

\bibitem{RV09}
M.~Rudelson and R.~Vershynin.
\newblock Smallest singular value of a random rectangular matrix.
\newblock {\em Comm. Pure Appl. Math.}, 62(12):1707--1739, 2009.

\bibitem{RV10:icm}
M.~Rudelson and R.~Vershynin.
\newblock Non-asymptotic theory of random matrices: extreme singular values.
\newblock In {\em Proceedings of the {I}nternational {C}ongress of
  {M}athematicians. {V}olume {III}}, pages 1576--1602. Hindustan Book Agency,
  New Delhi, 2010.

\bibitem{RV15}
M.~Rudelson and R.~Vershynin.
\newblock Small ball probabilities for linear images of high-dimensional
  distributions.
\newblock {\em Int. Math. Res. Not. IMRN}, (19):9594--9617, 2015.

\bibitem{Sha48}
C.E. Shannon.
\newblock A mathematical theory of communication.
\newblock {\em Bell System Tech. J.}, 27:379--423, 623--656, 1948.

\bibitem{Sta59}
A.J. Stam.
\newblock Some inequalities satisfied by the quantities of information of
  {F}isher and {S}hannon.
\newblock {\em Information and Control}, 2:101--112, 1959.

\bibitem{Tao10}
T.~Tao.
\newblock Sumset and inverse sumset theory for {S}hannon entropy.
\newblock {\em Combin. Probab. Comput.}, 19(4):603--639, 2010.

\bibitem{TV09:1}
T.~Tao and V.~Vu.
\newblock From the {L}ittlewood-{O}fford problem to the circular law:
  universality of the spectral distribution of random matrices.
\newblock {\em Bull. Amer. Math. Soc. (N.S.)}, 46(3):377--396, 2009.

\bibitem{Vil03:book}
C.~Villani.
\newblock {\em Topics in optimal transportation}, volume~58 of {\em Graduate
  Studies in Mathematics}.
\newblock American Mathematical Society, Providence, RI, 2003.

\bibitem{WM14}
L.~Wang and M.~Madiman.
\newblock Beyond the entropy power inequality, via rearrangements.
\newblock {\em IEEE Trans. Inform. Theory}, 60(9):5116--5137, September 2014.

\bibitem{WWM14:isit}
L.~Wang, J.~O. Woo, and M.~Madiman.
\newblock A lower bound on the {R{\'e}nyi} entropy of convolutions in the
  integers.
\newblock In {\em Proc. IEEE Intl. Symp. Inform. Theory}, pages 2829--2833.
  Honolulu, Hawaii, July 2014.

\bibitem{XMM17:isit:1}
P.~Xu, J.~Melbourne, and M.~Madiman.
\newblock {Infinity-R\'enyi Entropy Power Inequalities}.
\newblock In {\em Proc. IEEE Intl. Symp. Inform. Theory}, 2017, submitted.

\bibitem{XMM17:isit:2}
P.~Xu, J.~Melbourne, and M.~Madiman.
\newblock A min-entropy power inequality on groups.
\newblock In {\em Proc. IEEE Intl. Symp. Inform. Theory}, 2017, submitted.

\bibitem{Yu09:1}
Y.~Yu.
\newblock Monotonic convergence in an information-theoretic law of small
  numbers.
\newblock {\em IEEE Trans. Inform. Theory}, 55(12):5412--5422, 2009.

\end{thebibliography}

\end{document}